\documentclass[12pt]{article}
\usepackage{graphicx} % Required for inserting images
\usepackage[top=2.5cm,bottom=2.5cm,left=2.5cm,right=2.5cm]{geometry}

\usepackage{palatino}
\usepackage{latexsym}

\usepackage[mathscr]{eucal}
\usepackage{amsmath}
\usepackage{amsthm}
\usepackage{amsfonts}
\usepackage{amssymb}
\usepackage{amscd}
\usepackage{color}
\usepackage{graphicx}
\usepackage{graphics}
\usepackage{pifont}
\usepackage{subfigure}
\usepackage[makeroom]{cancel}
\usepackage[normalem]{ulem}
\usepackage[dvipsnames]{xcolor}
\usepackage{tikz}
\usepackage{cite}
\usepackage[hidelinks,pdfusetitle]{hyperref}

\theoremstyle{plain}
\newtheorem{theorem}{Theorem }

\newtheorem{lemma}[theorem]{Lemma}
\newtheorem{proposition}[theorem]{Proposition}
\newtheorem{definition}[theorem]{Definition}

\newcommand{\R}{{\mathbb R}}
\newcommand{\N}{{\mathbb N}}
\newcommand{\T}{{\mathbb T}}
\newcommand{\C}{{\mathbb C}}
\newcommand{\la}{{\langle}}
\newcommand{\ra}{{\rangle}}

\title{Approximate controllability in small times of bilinear Schrödinger equations with magnetic drift}

\begin{document}
\author{Eugenio Pozzoli\footnote{Univ Rennes, CNRS, IRMAR - UMR 6625, F-35000 Rennes, France (eugenio.pozzoli@univ-rennes.fr)}}
\maketitle

\abstract{%Recent advances on geometric control theory have shown that arbitrary $L^2$ local phases can be approximately controlled in small time for certain physically relevant Schrödinger PDEs on $\T^d$ and $\R^d$ \cite{duca-nersesyan,duca-pozzoli}. In this paper we build on such results and show that, under the same hypothesis, compositions with large families of diffeomorphisms (more precisely, the ones that can be written as a product of gradient flows) can be approximately controlled in small time, in $L^2$. In particular, in the one-dimensional case $d=1$, all compactly supported and isotopic to the identity diffeomorphisms can be obtained. 
%We show that small-time approximate controllability (STAC) of Schrödinger equations is stable under perturbation of the drift term by a large class of magnetic fields. 

We study the small-time approximate controllability of bilinear Schrödinger equations, where the drift is a magnetic Schrödinger operator and the control is an electric potential. We prove this property in two circumstances: (i) in $\R^d$, with a quadratic and an additional generic bounded electric potential in the control, and with a uniform magnetic field in the drift; (ii) in $\R^d$ or $\T^d$, with control electric potentials supported on a finite number of Hermite or Fourier eigenfunctions, and with any differentiable magnetic potential in the drift.

%We show that small-time approximate controllability (STAC) of Schrödinger equations with magnetic drift term.

}

\medskip



%\tableofcontents

% and show that the small-time approximate controllability problem for these Schrödinger PDEs posed on some boundaryless manifolds $M$ can be related to the controllability problem in the diffeomorphisms group of $M$.

%The main idea is that we can approximate the small-time behaviour of the solution of some controlled Schr\"odinger equations, with the solution of a transport equation along an arbitrary gradient vector field.

 %(more precisely, the ones that can be written as a product of gradient flows) 

% We then further apply these ideas and provide physically relevant examples of bilinear Schrödinger equations which are globally approximately controllable in arbitrarily small time.}

%%%%%%%%%%%%%%%%%%%%%%%%%%%%%%%%%%%%%%%%%%%%%%%%%%%%
\section{Introduction}
%---------------------------------------------------
\subsection{The model}
We consider the following initial value problem for magnetic Schrödinger equations controlled through an electric potential, on a smooth boundaryless Riemannian manifold $M$
\begin{equation}\label{eq:schro}
\begin{cases}
i\partial_t\psi(t,x)=\left(-\Delta_{A(x)}+V(x)+\sum_{j=1}^m u_j(t)W_j(x)\right)\psi(t,x), & (t,x) \in (0,T) \times M, \\
\psi(0,\cdot)=\psi_0.
\end{cases}
\end{equation}
The linear unbounded operator 
$$-\Delta_{A}:=
%-\sum_{j=1}^d(\partial_{x_j}-iA_j(x))^2=
-\Delta+|A|^2+i({\rm div}A+2\langle A,\nabla\rangle),$$
is called the magnetic Laplacian: here, $\Delta$ is the Laplace-Beltrami operator of $M$, ${\rm div}$ is the divergence operator computed w.r.t. the Riemannian volume form, $\langle \cdot,\cdot\rangle $ and $|\cdot|$ are respectively the Riemannian product and norm, $\nabla$ is the Riemannian gradient, and $A=(A_1,\dots,A_d)$ is a differentiable magnetic potential. The operator $\Delta_A$ is symmetric when defined on smooth compactly supported functions. The time-dependent electric potential $V(x)+\sum_{j=1}^m u_j(t)W_j(x)$ is possibly unbounded, and defined on a suitable domain. The time-independent magnetic Schrödinger operator $-\Delta_{A}+V$ is usually referred to as \emph{the drift}.

System \eqref{eq:schro} describes the dynamics of a quantum particle on a manifold $M$, subject to external electric fields with potentials $V,W_1,\dots,W_m$, and an external magnetic field with potential $A$. It is used to model a variety of physical situations, such as the Stark, resp. Zeeman, Hamiltonian for an atom in an electric, resp. magnetic, field (see, e.g., \cite[Chapter X.4]{rs2}).

We study the so-called quantum controllability problem. We think of the function $u(t)=(u_1(t),\dots,u_m(t))$ as a \emph{control function} that can be chosen and implemented in order to change the \emph{state} $\psi$ of the system. More precisely, given an initial and a final state $\psi_0,\psi_1$, we would like to find a control which steers the system from $\psi_0$ to $\psi_1$. We are in particular interested in the family of states that can be approximately reached from $\psi_0$ in arbitrarily small times. 

Assuming that we only consider piece-wise constant controls $u$, if the Hamiltonian is essentially self-adjoint on each time interval where the control is constant, the solution of \eqref{eq:schro} is well-defined at every time $t\geq 0$, from any $L^2$ initial state $\psi_0$. Denoting it by $\psi(t;u,\psi_0)$, we have that $\psi$ lives in the unitary sphere $\mathcal{S}$ of $L^2(M,\C)$,
\begin{equation} \label{def:S}
\mathcal{S}:=\{\psi\in L^2(M,\C)\, ;\, \|\psi\|_{L^2}=1\}.
\end{equation}

\begin{definition}
%[Small time $L^2$-approximate controllability] 
We say that \eqref{eq:schro} is \textbf{approximately controllable} if, for every 
$\psi_0, \psi_1\in \mathcal{S}$ and
$\varepsilon >0$, there exist 
a time $T \geq 0$,
a global phase $\theta \in [0,2\pi)$
and %a locally integrable 
a control $u\in PWC([0, T ], \R^m)$ such that 
%the Cauchy problem \eqref{eq:schro} 
%has a unique solution $\psi \in C^0([0,T],L^2(M))$, and
$\|\psi( T;u,\psi_0)-e^{i\theta} \psi_1\|_{L^2}<\varepsilon$.

We say that \eqref{eq:schro} is \textbf{small-time approximately controllable} if, for every 
$\psi_0, \psi_1\in \mathcal{S}$ and
$\varepsilon >0$, there exist 
a time $T \in[0,\varepsilon]$,
a global phase $\theta \in [0,2\pi)$
and %a locally integrable 
a control $u\in PWC([0, T ], \R^m)$ such that 
%the Cauchy problem \eqref{eq:schro} 
%has a unique solution $\psi \in C^0([0,T],L^2(M))$, and
$\|\psi( T;u,\psi_0)-e^{i\theta} \psi_1\|_{L^2}<\varepsilon$.
\end{definition}

%In quantum mechanics, a state is defined up to a global phase. Hence, controlling towards $e^{i\theta}\psi_1$ for some constant $\theta\in \R$ is physically the same as controlling towards $\psi_1$.
\subsection{Bibliographic comments}
In the seminal works \cite{beauchard1,beauchard-coron,coron,beauchard-laurent}, K. Beauchard, J.-M. Coron, and C. Laurent showed that equations of the form (\ref{eq:schro}), with $A=0$, are locally exactly controllable in $H_{(0)}^s(0,1)$ if and only if $s\geq 3$.

Later, approximate controllability of \eqref{eq:schro} (with $A=0$) has been proved to be generic in large time, independently and with different methods by V. Nersesyan \cite{nerse,nersesyan} and by U. Boscain, T. Chambrion, P. Mason, and M. Sigalotti \cite{BCMS,MS-generic}.

The capability of controlling quantum systems in small times has particularly relevant physical implications, both from a fundamental viewpoint and for technological applications. Indeed, those systems in reality suffer of extremely short lifespan before decaying. Their time-optimal control is in fact a crucial challenge in both physics and engineering (see e.g. the pioneering theoretical work \cite{khaneja-brockett-glaser} of N. Khaneja, S. J. Glaser, and R. Brockett on the minimal control-time for spin systems).

%The control of Schrödinger equations in small time represents today an open challenge in mathematical control theory. 

From a mathematical perspective, several progress has been obtained recently in the small-time control of Schrödinger PDEs such as \eqref{eq:schro} with $A=0$. 

K. Beauchard, J.-M. Coron, and H. Teismann showed that there exist examples of equations as \eqref{eq:schro} which are approximately controllable in large times, but not in small times \cite{beauchard-coron-teismann,beauchard-coron-teismann2}. They showed more precisely that Gaussian states are approximately preserved, at least for small times, when the control potential is linear and the drift potential is sub-quadratic. A. Duca and V. Nersesyan then proved that the approximate control between eigenstates on the torus is actually possible, by developing a Lie bracket strategy to change the phase of the wavefunction \cite{duca-nersesyan}. Recently, in collaboration with K. Beauchard, we showed that the previous examples studied in \cite{beauchard-coron-teismann,beauchard-coron-teismann2} are somehow singular and that small-time global approximate controllability is indeed possible in several circumstances \cite{beauchard-Pozzoli,beauchard-Pozzoli2} (we refer also to the works \cite{agrachev-bettina-eugenio} in collaboration with A. Agrachev and B. Kazandjian for an algebraic viewpoint and \cite{beauchard-carles-pozzoli} with K. Beauchard and R. Carles for nonlinear Schrödinger equations).

In the present paper, we extend those previous results of small-time approximate controllability to the case of Schrödinger equations \eqref{eq:schro} with magnetic drift term $A\neq 0$.

Some previous approximate controllability properties for magnetic Schrödinger equations were studied e.g. in \cite{chittaro-mason} by F. Chittaro and P. Mason (in large times and for uniform magnetic field in dimensions resp. $3$), in \cite{balsameda-etal} by A. Balsameda, D. Lonigro, and J. M. Pérez-Pardo (in large times and for constant magnetic potential in dimension $1$), in \cite{duca-pozzoli} by A. Duca and the author (in small times and for constant magnetic potential in any dimension $d$), and also by A. Duca, R. Joly, and D. Turaev in \cite{duca-joly-turaev} (in large times and for magnetic potential depending on transformations of the domain, in any dimension $d\geq 2$); in all those previous works, the magnetic field was used as a control operator instead of a drift operator. We also note that, if the Schrödinger drift operator $-\Delta_A+V$ has discrete spectrum (which is not supposed throughout this article), e.g. when $M$ is compact or either $A$ or $V$ are confining, the (large-time) approximate controllability of \eqref{eq:schro} could also be studied with finite-dimensional (Galerkin) approximation techniques developed in \cite{BCMS}.
%and more in general by furnishing informations on the set $\text{Adh}_{\mathcal{H}} \text{Reach}_{\text{st}}(\psi_0)$.

 \subsection{Results}

\subsubsection{Uniform magnetic field in $\R^d$}

 The first example of systems of the form \eqref{eq:schro} we study is posed in $\R^d, d\in \N$, with a differentiable magnetic potential tangent to all sphere of $\R^d$ centred at the origin, control on the frequency of a quadratic electric potential, and additional generic bounded control electric potential,
 \begin{equation}\label{eq:magnetic-schro}
 \begin{cases}
 i\partial_t\psi(t,x)=(- \Delta_{A(x)}+V(x) +u_{1}(t)|x|^2+u_2(t)W(x))\psi(t,x), (t,x)\in (0,T)\times \R^d\\
 \psi(\cdot,0)=\psi_0,
 \end{cases}
\end{equation}
where $A\in C^1(\R^d),W\in L^\infty(\R^d)$ and 
\begin{equation} \label{Hyp:V_transp}
V\in L^\infty_{\rm loc}(\R^d,\R) \quad \text{ and } \quad
  \exists a,b\geq 0,\, \forall x \in \R^d,\, V(x)\geq -a |x|^2-b.
\end{equation}
These assumptions guarantee the essential self-adjointess for all $u=(u_1,u_2)\in \R^{2}$ of the magnetic Schrödinger operator $- \Delta_{A}+V(x) +u_{1}|x|^2+u_2W(x)$ with domain $C^\infty_c(\R^d)$ on $L^2(\R^d)$ (see, e.g., \cite[Theorem 5.3]{shubin}). Therefore, for any $\psi_0\in L^2(\R^d),u\in PWC([0,\infty),\R^2)$, the (unique) solution of (\ref{eq:magnetic-schro}) is well-defined and \\$\psi(.;u,\psi_0) \in C^0([0,T],\mathcal{S})$.

 We prove the following result.
\begin{theorem}\label{thm:constant-magnetic}
%Let $M=\R^3,m=1$, $V\in L^\infty_{\rm loc}(\R^3)$, $W_0(x)=|x|^2$ and $A(x)=x\times b$ be a uniform magnetic field along the direction $b\in\R^3$. 
Let $A\in C^1(\R^d,\R^d)$ be such that 
$$\langle A(x),x\rangle=0, \forall x\in \R^d.$$ There exists a dense subset of control electric potentials $\mathcal{D}\subset L^\infty(\R^d)$ such that, for every $W\in\mathcal{D}$, system \eqref{eq:magnetic-schro} is small-time approximately controllable.%=(x_2b_3-x_3b_2,x_3b_1-x_1b_3,x_1b_2-x_2b_1). $$
 \end{theorem}
 Typical examples in dimensions $d=2,3$ of magnetic fields verifying $\langle A,x\rangle=0$ are given by uniform magnetic fields, i.e., $A(x_1,x_2)=b(-x_2,x_1), b\in \R,$ and $A(x)=x\times b, b\in \R^3$.
 
It is worth noting that the selection of an appropriate electric potential is necessary for controllability, as pointed out in the following counter-example.
\begin{theorem}\label{thm:obstruction}
Let $d=2$ and $A(x_1,x_2)=b(-x_2,x_1), b\in \R,$ or $d=3$ and $A(x)=x\times b, b\in \R^3$. Let also $V,W=0$. Then system \eqref{eq:magnetic-schro} is not approximately controllable.%\begin{equation}\label{eq:non-controllable}
% i\partial_t\psi(t,x)=(- \Delta_{A(x)} +u_{0}(t)|x|^2)\psi(t,x).
%\end{equation}
\end{theorem}
Such non-controllability is a consequence of the fact that Gaussian states are preserved by system \eqref{eq:schro} under the hypothesis of Theorem \ref{thm:obstruction}, as described in Section \ref{sec:thm:obstruction}.
%{\color{red}The idea is that $|A|^2=b^2|x|^2$, and that the flow of the vector field $A$ preserves the sphere, because $A$ points orthogonally to $x$. Moreover ${\rm div}A=0$ necessarily, as $A$ preserves the sphere. Hence, the Hamiltonian preserves the family of Gaussians $\{e^{iax^2}\sigma^{d/2}e^{-\sigma |x|^2}, a\in \R,\sigma>0\}$. What about more general $A$ such that $<A,x>=0$? And what if I add a control linear in $x$ (The obstruction should still hold I guess, e.g. in $d=2$ with $b=(1,1)$ we can do the computation explicitely: (the commutator $[A\nabla,x]=-x_2+x_1$ is proportional to another linear control in $x$) indeed the flow if a rotation of the circle hence the states $e^{ikx}$ goes into something like $e^{ik(x_1\cos\theta+x_2\sin\theta,-x_1\sin\theta +x_2\cos\theta)}$, hence Gaussians $\{e^{iax^2}e^{ibx}\sigma^{d/2}e^{-\sigma |x-c|^2}, a\in \R,\sigma>0, b,c\in\R^2\}$ are preserved or not)? }
\subsubsection{Any differentiable magnetic potential in $\R^d$ and $\T^d$}
In the second part of the article we study two additional examples of systems the form \eqref{eq:schro}, posed on $M=\T^d$ or $\R^d$, allowing for much more general magnetic perturbations in the drift term. In fact, we consider any differentiable magnetic potential.
%We shall also need an additional hypothesis on system \eqref{eq:schro}, introduced by Duca and Nersesyan in the recent work \cite{duca-nersesyan}.
%\begin{definition}
%We say that \eqref{eq:schro} satisfies the small-time approximate controllability of phases if, $\forall \varphi\in L^\infty_{\rm loc}(M,\R), \psi_0\in L^2(M,\C)$, the state $e^{i\varphi}\psi_0$ is small-time approximately reachable from $\psi_0$.
%\end{definition}
%\begin{theorem}\label{thm:general-magnetic}
%Let $M=\R^d,\T^d$ and suppose that $\Delta_A-V$ is self-adjoint. satisfy the small-time $L^2$-approximate controllability of phases. Then, system \eqref{eq:magnetic-schro} is small-time $L^2$-approximately controllable.
%\end{theorem}
%We can furnish two explicit examples of system satisfying the small-time $L^2$-approximate controllability of phases.
%\paragraph{An equation posed on $M=\T^d=\R^d/2\pi\Z^d$.}

The first example is the equation 
\begin{equation}\label{eq:torus}
\begin{cases}
i\partial_t\psi(t,x)=\!\!\Big(\!\!-\Delta_{A(x)}+V(x
)+\sum\limits_{j=1}^d ( u_{2j-1}(t)\sin +u_{2j}(t)\cos)\langle b_j , x\rangle \Big)\psi(t,x),\,\,x\in\T^d,\\
\psi(0,\cdot)=\psi_0, &
\end{cases}
\end{equation}
where 
%$(A_1,\dots,A_d)\in C^1(\T^d,\R^d), 
$V\in L^\infty(\T^d,\R), A\in C^1(\T^d,\R^d)$, and 
\begin{equation} \label{Def:ej}
b_1=(1,0,\dots,0),\quad b_2=(0,1,0,\dots,0),\quad \dots, \quad b_{d-1}=(0,\dots,0,1,0),\quad b_d=(1,\dots,1).
\end{equation}
In the one-dimensional case $d=1$, it describes the orientation in the plane of a rigid molecule subject to a drift magnetic field, controlled by the dipolar interactions with two electric fields of constant orthogonal directions and time-variable amplitudes. %It is widely used in physics and chemistry as a model for rotational molecular dynamics (see, e.g., the recent review \cite{koch} and the references therein). 
%Recently, small-time approximate controllability properties of \eqref{eq:torus} were obtained in \cite{duca-nersesyan,coron-xiang-zhang}. 

The second example is the equation 
\begin{equation}\label{eq:line}
\begin{cases}
i\partial_t\psi(t,x)=\left( -\Delta_{A(x)}+V(x)+\sum\limits_{j=1}^d u_{j}(t)x_j+u_{d+1}(t)e^{-|x|^2/2}  \right)\psi(t,x),& x \in  \R^d, \\
\psi(0,\cdot)=\psi_0, &  
\end{cases}
\end{equation}
where  $A\in C^1(\R^d,\R^d)$ and $V$ verifies \eqref{Hyp:V_transp}. It models the dipolar interaction of a quantum particle in a drift magnetic field with controls coupling to its positions $x_j$, and an additional control concentrated around the origin as a Gaussian function. Notice that such Hamiltonian, when e.g. $A,V\in L^\infty$, has continuous spectrum for all $u\in \R^{d+1}$ (see e.g. \cite[Theorem XIII.15]{rs4}). 
%Recently, small-time approximate controllability properties of (\ref{eq:line}) were obtained in \cite{duca-pozzoli}. 

\medskip

The assumption $A\in C^1(\T^d),V\in L^\infty(\T^d,\R)$ (resp. $A\in C^1(\R^d)$ and $V$ verifies \eqref{Hyp:V_transp}) guarantees the essential self-adjointess for all $u\in \R^{m}$ of the magnetic Schrödinger operators $-\Delta_A+V+\sum\limits_{j=1}^d ( u_{2j-1}\sin +u_{2j}\cos)\langle b_j , x\rangle$ with domain $C^\infty(\T^d)$ (resp. $-\Delta_A+V+\sum\limits_{j=1}^d u_{j}x_j+u_{d+1}e^{-|x|^2/2} $ with domain $C^\infty_c(\R^d)$) on $L^2(\T^d)$ (resp. $L^2(\R^d)$) (see, e.g., \cite[Theorem 5.2]{shubin} for $\T^d$ and \cite[Theorem 5.3]{shubin} for $\R^d$). Therefore, for any \\$u\in PWC([0,\infty),\R^m)$, the (unique) solutions of (\ref{eq:torus}) and (\ref{eq:line}) are well-defined and $\psi(.;u,\psi_0) \in C^0([0,T],\mathcal{S})$.
Our main second result is the following.
\begin{theorem}\label{thm:examples}
Let $A\in C^1(M,\R^d)$ where $M=\T^d,$ and $\R^d$. Systems \eqref{eq:torus} and \eqref{eq:line} are small-time approximately controllable.
\end{theorem}

The strategy used in this paper follows geometric methods, combining Lie brackets arguments à la \emph{Agrachev-Sarychev} \cite{agrachev-sarychev} recently adapted to Schrödinger equations in the works \cite{duca-nersesyan} by A. Duca and V. Nersesyan, and \cite{coron-xiang-zhang} by J.-M. Coron, S. Xiang, and P. Zhang, and the control of diffeomorphisms through gradient vector fields developed in \cite{beauchard-Pozzoli} and \cite{beauchard-Pozzoli2} by K. Beauchard and the author. The convergences are harder to obtain here, due to the presence of the additional magnetic perturbation in the drift Schrödinger operator. Intuitively, being the magnetic perturbation $|A|^2+i({\rm div}A+2\langle A,\nabla\rangle)$ a first order differential operator, it is dominated in the Lie bracket asymptotics by the second order differential operator $-\Delta$. Exploiting such fact, we prove that it does not affect the small-time approximate controllability which holds true when $A=0$ (as shown in \cite{beauchard-Pozzoli,beauchard-Pozzoli2} by K. Beauchard and the author). 

\subsection{Structure of the article}
The article is organised as follows: in Section \ref{sec:control-notions} we introduce some control notions; in Section \ref{sec:thm:obstruction} we give a proof of Theorem \ref{thm:obstruction}; in Section \ref{sec:uniform} we prove Theorem \ref{thm:constant-magnetic}; in Section \ref{sec:gradient} we show that the control of phases implies the control of gradient vector fields; in Section \ref{sec:phase} we give a proof of Theorem \ref{thm:examples}.

\section{Some control notions}\label{sec:control-notions}

{\color{black}\begin{definition}  \label{def:L2AR}
Let $T \in \R_+$ and $L$ be a unitary operator on $L^2(M,\C)$. 
\begin{itemize}
\item $L$ is $L^2$-exactly reachable in time $T$ if, for every $\psi_0 \in  \mathcal{S}$, there exist $u \in PWC(0,T)$ and $\theta\in[0,2\pi)$ such that 
 $ \psi(T;u,\psi_0) =  e^{i\theta}L \psi_0 $. 
\item $L$ is $L^2$-approximately reachable in time $T^+$ if, for every $\psi_0 \in \mathcal{S}$ and $\varepsilon>0$, there exist $T_1 \in [T,T+\varepsilon]$, $u \in PWC(0,T_1)$ and $\theta \in [0,2\pi)$ such that $\| \psi(T_1;u,\psi_0)-  e^{i\theta}L \psi_0 \|_{L^2}<\varepsilon$.
\item $L$ is $L^2$-STAR if it is $L^2$-approximately reachable in time $0^+$.
\end{itemize}
\end{definition}

The following property was shown in \cite[Lemma 7]{beauchard-Pozzoli}, and is used extensively in the rest of the article.

\begin{lemma} \label{lem:reachable-operators-bis}
\begin{enumerate}
    \item For $j=1,2$, let $T_j \in \R_+$ and $L_j$ be a unitary operator on $L^2(M,\C)$ which is $L^2$-approximately reachable in time $T_j^+$. Then the operator $L_1 L_2$ is $L^2$-approximately reachable in time $(T_1+T_2)^+$.
    \item Let $(T_n)_{n\in\N}$ be a decreasing sequence of $\R_+$ and $T:=\lim_{n\to \infty} T_n$. For every $n \in \N$, let $L_n$ be a unitary operator on $L^2(M,\C)$ which is $L^2$-approximately reachable in time $T_n^+$. We assume $L_n$ strongly converges to $L$. Then $L$ is $L^2$-approximately reachable in time $T^+$.
\end{enumerate}
\end{lemma}}

\begin{definition}
\label{def:Diffc0_action}
For $P\in {\rm Diff}(M)$, the unitary operator on $L^2(M,\C)$ associated with $P$ is defined by 
\begin{equation} \label{Def:LP}
\mathcal{L}_{P}\psi =  J_P^{1/2}(\psi\circ P),
\end{equation}
where $J_P := \text{det}(DP)$ is the determinant of the Jacobian matrix $DP$ of $P$. Then $\|\mathcal{L}_{P}\psi\|_{L^2}=\|\psi\|_{L^2}.$
\end{definition}
An example of unitary operator of this form in $\R^d$ is given in the following definition.
\begin{definition} 
For $\alpha \in \R^*$, the dilation $D_{\alpha}$
is the linear isometry of $L^2(\R^d,\C)$ defined by $(D_{\alpha}\psi)(x)=|\alpha|^{d/2}\psi(\alpha x)$.
\end{definition}

\begin{definition} \label{def:flow}
${\rm Vec}(M)$ denotes the space of 
globally Lipschitz
smooth vector fields on $M$.

For $f \in {\rm Vec}(M)$, $\phi_f^s$
denotes the flow associated with $f$ at time $s$:
for every $x_0 \in M$, $x(s)= 
\phi_f^s(x_0)$ is the solution of
the ODE $\dot{x}(s)=f(x(s))$ associated with the initial condition $x(0)=x_0$.
 \end{definition}

For $f \in {\rm Vec}(M), t\in \R$, then $\phi_f^t \in {\rm Diff}(M)$.

\begin{definition} \label{Def:STC}
We introduce the following small-time controllability (STC) notions for \eqref{eq:schro}:
\begin{itemize}
\item \textbf{STC of phases}: for every $\varphi \in L^\infty_{\rm loc}(M,\R)$, the operator $e^{i\varphi}$ is $L^2$-STAR.
\item For a %homogeneous 
subset $\frak{G}$ of ${\rm Vec}(M)$, 
%(i.e., $f\in \frak{G}\Rightarrow \lambda f\in \frak{G}, \forall \lambda\in \R$), 
the \textbf{STC of flows of vector fields in $\frak{G}$} 
holds if
for every $f \in \frak{G}$ and $t \in \R$, the operator $\mathcal{L}_{\phi_f^t}$ is $L^2$-STAR for \eqref{eq:schro}. 
\end{itemize}
\end{definition}

We remark that, for any $f\in {\rm Vec}(M)$, the (skew-adjoint) first-order differential operator $$\mathcal{T}_f := \langle f , \nabla \cdot \rangle + \frac{1}{2} {\rm div}(f)$$ is the generator of the $L^2$-unitary group $\{\mathcal{L}_{\phi^t_f}\}_{t\in \R}$, that is, $e^{t\mathcal{T}_f}=\mathcal{L}_{\phi^t_f}$ (thanks to the method of characteristics).

\section{Proof of Theorem \ref{thm:obstruction}}\label{sec:thm:obstruction}
Let
$${\rm R}(\psi_0)=\{\psi_1\in \mathcal{S}|\exists T\geq 0, u\in {\rm PWC}(0,T)\text{ s.t. }\psi(T;u,\psi_0)=\psi_1\}$$
be the set of reachable states from $\psi_0\in \mathcal{S}$. Notice that, for $A=b(-x_2,x_1)$ or $A=x\times b,$ 
\begin{align}
& {\rm div}(A)=0, |A|^2=|b|^2|x|^2,\label{eq:divergence-free}\\
& \langle A(x),x\rangle=0, \forall x\in\R^d \text{ (hence } |\phi_{A}^s(x)|^2=|x|^2, \forall s\in \R\text{)}.\label{eq:preserving-spheres}
 \end{align}
Relation \eqref{eq:preserving-spheres} says that the vector field $A$ is tangent to any sphere $\{|x|=a\},a>0$, hence its flow is an isometry of the euclidean space. Thanks to the Trotter-Kato product formula (see e.g. \cite[Theorem VIII.30]{rs1}) and properties \eqref{eq:divergence-free}, $\overline{{\rm R}(\psi_0)}^{L^2}$ is then equal to the $L^2$-closure of the set
$$\left\{\prod_{j=1}^N e^{it_j\Delta}e^{t_j\langle A,\nabla\rangle}e^{i(u_j+t_j|b|^2)|x|^2}\psi_0\mid t_j\geq 0,u_j\in \R, N\in\N\right\}.$$
Let 
\begin{equation} \label{def:Gauss}
\mathcal{G}:= \{ x \in \R^d \mapsto Ce^{i \theta}   a^{\frac{d}{4}} e^{-(a+ib)|x|^2}  ; \theta \in [0,2\pi), a>0, b \in \R \},
\end{equation}
$C:=(8/\pi)^{d/2}$, be the set of centred normalized Gaussian states. The propagator $e^{it|x|^2}, t\in\R,$ clearly preserves $\mathcal{G}$; also $e^{it\Delta}, t\in\R,$ preserves $\mathcal{G}$ thanks to the explicit form of the solution of the free Schrödinger equation in $\R^d$; and finally also $e^{t\langle A,\nabla\rangle}, t\in \R$ preserves $\mathcal{G}$, thanks to \eqref{eq:preserving-spheres} and the equality $(e^{t\langle A,\nabla\rangle}\psi_0)(x)=\psi_0(\phi_A^t(x)).$
Hence, 
$${\rm R}(\psi_0)\subset \mathcal{G},\quad \forall \psi_0\in  \mathcal{G}.$$
The conclusion follows by the fact that $\mathcal{G}$ is a closed subspace of the $L^2$-unit sphere $\mathcal{S}$ (see, e.g., \cite[Lemma 29]{beauchard-Pozzoli}).

\section{Proof of Theorem \ref{thm:constant-magnetic}}\label{sec:uniform}

Let us consider the subsystem
\begin{equation}\label{eq:trap_VV}
\begin{cases}
i\partial_t\psi(t,x)=(- \Delta _{A}+ V(x) + u(t)|x|^2)\psi(t,x),& (t,x) \in (0,T)\times\R^d, \\
\psi(0,\cdot)=\psi_0.&
\end{cases}
\end{equation}
In order to establish Theorem \ref{thm:constant-magnetic}, we shall need the following result.
\begin{proposition} \label{Prop:exp(isDelta)}
Let $V$ satisfy \eqref{Hyp:V_transp}, and $A\in C^1(\R^d)$ be such that $\langle A(x),x\rangle=0,\forall x\in\R^d$.
For system (\ref{eq:trap_VV}), the following operators are $L^2$-STAR:
\begin{enumerate}
    \item $e^{i \delta |x|^2}$ for every $\delta\in\R$,
    \item $D_{\alpha}$ for every $\alpha > 0$,
    \item $e^{i \sigma \Delta}$ for every $\sigma \geq 0$,
    \item $e^{i \sigma (\Delta-|x|^2)}$ for every $\sigma \geq 0$.
    \end{enumerate}
\end{proposition}

\begin{proof}

The notations used to prove the Item $i$ are not valid in the proof of the Item $j \neq i$. 

\medskip

\noindent \emph{1.} Let $\delta\in\R^*$. For $\tau>0$, the operator $\tau (\Delta_A-V)+ \delta |x|^2$ is essentially self-adjoint on $C^{\infty}_c(\R^d,\C)$ (we apply e.g. \cite[Theorem 5.3]{shubin} with $Q(x)=|x|^2+1$), thus its closure $L_{\tau}$ is self-adjoint. 
The operator $L_{0}:=\delta |x|^2$ is self-adjoint on $D(A_0):=\{ \psi \in L^2(\R^d,\C) ; |x|^2 \psi \in L^2(\R^d,\C) \}$. $C^{\infty}_c$ is a common core to $L_0$ and $L_{\tau}$. For any $\psi \in C^{\infty}_c(\R^d,\C)$,
$\|(L_{\tau}-L_0)\psi\|_{L^2} = \tau \| (\Delta_A-V) \psi \|_{L^2} \to 0$ as $\tau \to 0$. Thus, for every
$\psi \in L^2(\R^d,\C)$ 
$\|e^{i L_{\tau}} \psi - e^{i L_0} \psi \|_{L^2} \rightarrow 0$ as $\tau \to 0$. 
Moreover, for every $\tau>0$, the operator $e^{i L_{\tau}}$ is $L^2$-exactly reachable in time $\tau$ because associated with the constant control $u=-\delta/\tau$.
Hence, the operator 
$e^{i L_0}$ is $L^2$-STAR.

\medskip

\noindent \emph{2.} Let $\alpha>0$. For every $\tau>0$, the operator
$$L_{\tau} := e^{i\frac{\log(\alpha) |x|^2}{4\tau}}
e^{i\tau\left(\Delta_A-V+\frac{\log^2(\alpha) |x|^2}{4\tau^2}\right)}
e^{-i\frac{\log(\alpha) |x|^2}{4\tau}}$$
is $L^2$ approximately reachable in time $\tau^+$, 
by Item 1.

\medskip

\noindent \emph{Step 1: We prove that, for every $\tau>0$,
$$L_{\tau} = \exp\left( i\tau(\Delta_A-V)+\frac{\log(\alpha)}{2}\left(d+2\la x,\nabla\ra\right)
 \right). $$}
The operator $\tau (\Delta_A-V+ \log^2(\alpha) |x|^2/ (4\tau^2))$ is essentially self-adjoint on $C^\infty_c(\R^d,\C)$, thus its closure $A_{\tau}$ is self-adjoint. The operator $B=\log(\alpha) |x|^2 / (4\tau)$ is self-adjoint on $D(B):=\{ \psi \in L^2(\R^d,\C) ; |x|^2 \psi \in L^2(\R^d,\C) \}$. $C^{\infty}_c(\R^d,\C)$ is a core of $A_{\tau}$ and $e^{iB}$ is an isomorphism of $C^{\infty}_c(\R^d,\C)$. Thus, we have
$$L_{\tau}= \exp \left( i\tau
e^{i\frac{\log(\alpha) |x|^2}{4\tau}}
\left(\Delta_A-V+\frac{\log^2(\alpha) |x|^2}{4\tau^2}\right)
e^{-i\frac{\log(\alpha) |x|^2}{4\tau}}
\right).$$
Moreover, direct computations prove
\begin{align} 
i\tau
e^{i\frac{\log(\alpha) |x|^2}{4\tau}}
\left(\Delta_A-V+\frac{\log^2(\alpha) |x|^2}{4\tau^2}\right)
e^{-i\frac{\log(\alpha) |x|^2}{4\tau}}
=&
i \tau(\Delta_A-V)+\frac{\log(\alpha)}{2}\left(d+2\la x,\nabla\ra\right) \nonumber\\
&-\frac{i\log^2(\alpha)}{2}\langle A,x\rangle \nonumber\\
=&i \tau(\Delta_A-V)+\frac{\log(\alpha)}{2}\left(d+2\la x,\nabla\ra\right),\label{log(Ltau)} 
\end{align}
since $\langle A,x\rangle=0$.
 
\medskip

\noindent \emph{Step 2: We prove that, for every $\psi \in L^2(\R^d,\C)$, $\|(L_{\tau}-D_{\alpha})\psi\|_{L^2} \to 0$ as 
$\tau \to 0$.}
By (\ref{log(Ltau)}),
for $\tau>0$,
the operator $\tau(\Delta_A-V)-i (d+2\la x,\nabla\ra)\log(\alpha)/2$ is essentially self-adjoint on $C^\infty_c(\R^d,\C)$, thus its closure $B_{\tau}$ is self-adjoint. The operator 
$B_0:=-i (d+2\la x,\nabla\ra)\log(\alpha)/2$ is self-adjoint on $D(B_0):=\{ \psi \in L^2(\R^d,\C) ; \langle x , \nabla \psi \rangle \in L^2(\R^d,\C)  \}$. $C^{\infty}_c(\R^d,\C)$ is a common core of $B_{\tau}$ and $B_0$. For every $\psi \in C^\infty_c(\R^d,\C)$,
$\| (B_{\tau}-B_0) \psi \|_{L^2} = \tau \|(\Delta-V) \psi \|_{L^2} \to 0$ as $\tau \to 0$. 
Thus, by Step 1, for every $\psi \in L^2(\R^d,\C)$,
$$\| (L_{\tau}-D_{\alpha})\psi\|_{L^2} = \| (e^{iB_{\tau}} - e^{i B_0} ) \psi \|_{L^2}
\underset{\tau \to 0}{\longrightarrow} 0.$$
Finally, $D_{\alpha}$ is $L^2$-STAR.

\medskip

\noindent \emph{3.} Let $\sigma >0$. For every $t>0$, the operator
$$L_{t} := D_{t^{1/2}}e^{i\sigma t(\Delta_A-V)}D_{t^{-1/2}} $$
is $L^2$-approximately reachable in time $(\sigma t)^+$, by Item \emph{2}. Moreover,
$$L_t=\exp(D_{t^{1/2}}i\sigma t (\Delta_A-V)D_{t^{-1/2}}).$$ 
Then, for every $\psi\in C^\infty_c(\R^d,\C)$, explicit dilations and differentiations give the action of the generator: 
\begin{align}
&D_{t^{1/2}}i\sigma t (\Delta_A-V)D_{t^{-1/2}}\psi\nonumber\\
=&i\sigma(\Delta-tV(t^{1/2}x)-t|A|^2(t^{1/2}x)-it({\rm div}A)(t^{1/2}x)-2it^{1/2}\langle A(t^{1/2}x),\nabla\rangle)\psi.\label{eq:rescaled}
\end{align}
For every $t>0$, the operator defined in \eqref{eq:rescaled} is essentially self-adjoint on $C^{\infty}_c(\R^d,\C)$ thus its closure $A_t$ is self-adjoint.
The operator $A_0:=\sigma \Delta$ is self-adjoint on $D(A_0):=H^2(\R^d,\C)$. $C^{\infty}_c(\R^d,\C)$ is a common core to $A_0$ and $A_{t}$. Let $\psi \in C^\infty_c(\R^d,\C)$ and $R>0$ such that $\text{supp}(\psi) \subset K:=B_{\R^d}(0,R)$. We have
\begin{align*}
&\|(A_{t}-A_0)\psi\|^2_{L^2}   \\
=&\| (tV(t^{1/2}\cdot)+t|A|^2(t^{1/2}\cdot)+it({\rm div}A)(t^{1/2}\cdot)+2it^{1/2}\langle A(t^{1/2}\cdot),\nabla\rangle)\psi \|^2_{L^2}  \\
=& \int_{K} |[tV(t^{1/2} x)+t|A|^2(t^{1/2}x)+it({\rm div}A)(t^{1/2}x)] \psi(x)+2it^{1/2}\langle A(t^{1/2}x),\nabla\psi(x)\rangle)|^2 dx 
\\  \leq&
ct  \int_{K}  |\psi(x)|^2 dx
+c't^{1/2}\int_{K}|\nabla\psi(x)|^2 dx
 \underset{t \to 0}{\longrightarrow} 0,
\end{align*}
where $c,c'$ are some positive constants (thanks to the fact that $A,V\in L^\infty_{\rm loc}$). This convergence proves that, 
for every $\psi \in L^2(\R^d,\C)$,
$\|(L_{t} - e^{i\sigma \Delta})\psi  \|_{L^2} 
=\| (e^{i A_t} - e^{i A_0}) \psi \|_{L^2}
\underset{t \to 0}{\longrightarrow} 0$; hence, $e^{i \sigma \Delta}$ is $L^2$-STAR.

\medskip

\noindent \emph{4.} For every $n \in \N^*$, the operator
$$L_n := \left( e^{i \frac{\sigma}{n} \Delta} e^{-i \frac{\sigma}{n}|x|^2} \right)^n $$
is $L^2$-STAR by Items 1 and 3.
The operator $A:=\sigma \Delta$ is self-adjoint on $D(A):=H^2(\R^d,\C)$. The operator $B:=-\sigma |x|^2$ is self-adjoint on $D(B):=\{ \psi \in L^2(\R^d,\C) ; |x|^2 \psi \in L^2(\R^d,\C) \}$. The operator $A+B=\sigma(\Delta-|x|^2)$ is self-adjoint on $D(A) \cap D(B)$. Thus, by Lie-Trotter-Kato product formula, for every $\psi \in L^2(\R^d,\C)$,
$\| (L_n-e^{i \sigma(\Delta-|x|^2)})\psi \|_{L^2} =
\| (e^{iA/n}e^{iB/n})^n \psi - e^{i(A+B)}\psi  \|_{L^2} \to 0 $ as $\tau \to 0$. This proves that $e^{i \sigma(\Delta-|x|^2)}$ is $L^2$-STAR.
\end{proof}

Theorem \ref{thm:constant-magnetic} then follows exactly as in \cite{beauchard-Pozzoli}.

 \section{Control of flows of gradient vector fields}\label{sec:gradient}

 In this section, we prove the STC of flows of gradient vector fields as a consequence of the STC of phases.
\begin{proposition}\label{prop:gradient-control}
Let $M=\T^d$ (resp. $M=\R^d$), $V,W_1,\dots,W_m\in L^\infty(\T^d), A\in C^1(\T^d)$ (resp. $A\in C^1(\R^d)$ and $V,W_1,\dots,W_m$ satisfy \eqref{Hyp:V_transp}). For system \eqref{eq:schro}, the STC of phases implies the STC of vector fields in $\frak{G}$, where $\frak{G}$ is the following subset of ${\rm Vec}(M)$ (made of gradient vector fields)
\begin{equation} \label{def:G_gothique}
    \frak{G}\!\!:=\!\!\{\nabla \varphi\!\mid\! \varphi\in C^\infty (M,\R) \text{ and either } \varphi \in  L^2(M,\R),\nabla\varphi \in W^{1,\infty}\cap L^4(M),  \text{or }  \nabla\varphi=\text{const.} \},
\end{equation}
\end{proposition}
\begin{proof}
Let $f  \in \frak{G}$ with $f:=2 \nabla \varphi$.
Let $n \in \N^*$ and $\tau>0$. The operator
$$
L_{\tau,n}:= 
\left( 
e^{i \frac{|\nabla \varphi|^2}{n\tau}  }
e^{i  \frac{\varphi}{\tau} }
e^{i\frac{\tau}{n}(\Delta_A-V)}
e^{-i \frac{\varphi}{\tau} }
\right)^{n}
$$
is $L^2$-approximately reachable in time $\tau^+$.

\medskip

\noindent \emph{Step 1: We prove that
\begin{equation} \label{Step1/Ltaun}
L_{\tau,n} =
\left( 
e^{ i \frac{|\nabla \varphi|^2}{n \tau }  }
\exp \left( i\frac{\tau}{n}\,
e^{i \frac{\varphi}{\tau}  }
(\Delta_A-V)
e^{-i \frac{\varphi}{\tau} }
\right)
\right)^{n}.
\end{equation}}
The operator $(\Delta_A-V)$ is essentially self-adjoint on $C^{\infty}_c(M,\C)$ thanks to the assumptions on $A,V$ (see, e.g., \cite[Theorem 5.2]{shubin} for $\T^d$ and \cite[Theorem 5.3]{shubin} for $\R^d$), thus its closure $A$ is self-adjoint. The operator $B:=\varphi/\tau$ is self-adjoint on $D(B):=\{ \psi \in L^2(M,\C) ; \varphi \psi \in L^2(M,\C) \}$.  The operator $e^{i B}$ is an isomorphism of $C^{\infty}_c(M,\R)$ because $\varphi \in C^{\infty}(M)$. 
Thus, (\ref{Step1/Ltaun}) is proved.

\medskip

\noindent \emph{Step 2: We prove that $L_{\tau}$ is $L^2$-reachable in time $\tau^+$, where}
$$L_{\tau}:= 
%\exp \left( 
%i \tau (\Delta-V)+2 \langle \nabla \varphi , \nabla \rangle + \Delta %\varphi 
%\right)= 
\exp \left( 
i \tau (\Delta_A-V) + \mathcal{T}_f+2i\langle A,\nabla\varphi\rangle
\right). $$
The operator 
$e^{i \frac{\varphi}{\tau}  }
(\Delta_A-V)
e^{-i \frac{\varphi}{\tau} }$
is essentially self-adjoint on $C^{\infty}_c(M,\C)$ by hypothesis, thus its closure $A_1$ is self-adjoint. The multiplicative operator $B_1:=|\nabla \varphi|^2/\tau$ is essentially self-adjoint on $C^\infty_c(M,\C)$ because $|\nabla \varphi|\in L^\infty(M,\R)$. Hence, $A_1+B_1$ is self-adjoint on $D(A_1)$ because $B_1$ is bounded (we can thus apply Kato-Rellich Theorem). Thus, 
for every $\psi \in L^2(M,\C)$,
$\| (L_{\tau,n}-L_{\tau}')\psi \|_{L^2} \underset{n \to \infty}{\longrightarrow} 0$
where
$$L_{\tau}':=\exp \left( i  \left\{  \frac{|\nabla \varphi|^2}{\tau}  + \tau 
e^{i \frac{\varphi}{\tau}  }
(\Delta_A-V)
e^{-i \frac{\varphi}{\tau}}
  \right\} \right),
  $$
  and direct computations prove that
\begin{equation} \label{Step2:Log}
      \tau 
e^{i \frac{\varphi}{\tau}  }
(\Delta_A-V)
e^{-i \frac{\varphi}{\tau}}
+ \frac{|\nabla \varphi|^2}{\tau} 
=
\tau (\Delta-V) - i \mathcal{T}_f+2\langle A,\nabla\varphi\rangle,
\end{equation}
thus $L_{\tau}'=L_{\tau}$. The operator $L_{\tau}$ is thus $L^2$-approximately reachable in time $\tau^+$. 

\medskip

\noindent \emph{Step 3: We prove that the operator 
$e^{(\mathcal{T}_f+2i\langle A,\nabla\varphi\rangle)}$ is $L^2$-STAR.}
The operator (\ref{Step2:Log}) defined on $C^{\infty}_c(M,\C)$ has a self-adjoint closure $A_{\tau}$ (we apply Kato-Rellich Theorem with $A=\tau 
e^{i \frac{\varphi}{\tau}  }
(\Delta_A-V)
e^{-i \frac{\varphi}{\tau}}
+ \frac{|\nabla \varphi|^2}{\tau}$ and the bounded operator $B=-\frac{|\nabla \varphi|^2}{\tau}$).
The operator
$A_0:=- i \mathcal{T}_f+2\langle A,\nabla\varphi\rangle$
is self-adjoint because $\nabla \varphi$ is globally Lipschitz continuous: indeed, the method of characteristics give the explicit formula (see also \cite{chernoff})
$$\left( e^{t (\mathcal{T}_f+2i\langle A,\nabla\varphi\rangle)} \psi \right)(x)= 
\psi( 
\phi_f^t(x)
%\mathfrak{e}^{tf} x 
) e^{\frac{1}{2} \int_0^t {\rm div} f( 
\phi_f^s(x))
%\frak{e}^{sf} x
+2i \langle A,\nabla \varphi\rangle(\phi_f^s(x)) ds}.$$
Moreover, $C^{\infty}_c(M,\C)$ is a common core of $A_{\tau}$ and $A_0$.
%because
%$\nabla \varphi \in L^{\infty}(\R^d,\R)$.
For every $\psi \in C^\infty_c(M,\C)$,
$\| (A_{\tau}-A_0) \psi \|_{L^2} =
\tau \left\| (\Delta_A-V)\psi \right\|_{L^2}    
\to 0$ as $\tau \to 0$.
Thus, for every $\psi \in L^2(M,\C)$,
$\| 
(L_{\tau} - e^{ \mathcal{T}_f +2i\langle A,\nabla\varphi\rangle}) \psi
\|_{L^2} 
= \| (e^{i A_{\tau}} - e^{i A_0}) \psi \|_{L^2}
\to 0$ as $\tau \to 0$.
Finally, the operator $e^{\mathcal{T}_f+2i\langle A,\nabla\varphi\rangle}$ is $L^2$-STAR.  

\medskip

\noindent \emph{Step 3: We prove that the operator 
$e^{\mathcal{T}_f}$ is $L^2$-STAR.} Both the operators $e^{\mathcal{T}_f+2i\langle A,\nabla\varphi\rangle}$ and $e^{-2i\langle A,\nabla\varphi\rangle}$ are STAC (indeed $\langle A,\nabla\varphi\rangle\in L^\infty_{loc}$ being $A\in C^1$ and $\nabla\varphi\in L^\infty$). Moreover, both the operators $-i\mathcal{T}_f+2\langle A,\nabla\varphi\rangle$ and $-2\langle A,\nabla\varphi\rangle$ are essentially self-adjoint on $C^\infty_c$, and their sum $-i\mathcal{T}_f$ is essentially self-adjoint on $C^\infty_c$. By Lie-Trotter-Kato product formula, $e^{\mathcal{T}_f}$ is STAC.

\end{proof}

\section{Proof of Theorem \ref{thm:examples}} \label{sec:phase}

In this section, we prove that systems \eqref{eq:torus} and \eqref{eq:line} satisfy the small-time approximate controllability of phases.

%---------------------------------------------------------------
\subsection{On $\T^d$}

\begin{proposition} \label{Prop:DN}
Let $V \in L^\infty(\T^d,\R)$.
System \eqref{eq:torus} satisfies the following property:
for every $\varphi \in L^2(\T^d,\R)$,
the operator
$e^{i \varphi}$ is $L^2$-STAR.
\end{proposition}

\begin{proof}
\noindent \emph{Step 1: We prove that, for every $\varphi \in \mathcal{H}_0:=\text{span}_{\R}\{\sin \langle b_j,x\rangle, \cos \langle b_j , x \rangle; j\! \in\! \{1,\dots,d\}\! \}$, the operator $e^{i \varphi}$ is $L^2$-STAR.}
Let $\alpha \in \R^{2d}$ and 
$\varphi: x \in \T^d\! \mapsto \!\!\sum_{j=1}^{d} (\alpha_{2j-1} \sin + \alpha_{2j} \cos) \langle b_j,x\rangle$. For any $\tau>0$, the operator
$L_{\tau}:=e^{ i\tau(\Delta_A-V) + i\varphi }$
is $L^2$-exactly reachable in time $\tau$,
because associated with the constant controls $u_j=-\alpha_j/\tau$ applied on the time interval $[0,\tau]$.
For $\tau>0$, the operator $A_{\tau}:=\tau(\Delta_A-V)+ \varphi$ is self-adjoint on $D(A_{\tau}):=H^2(\T^d,\C)$,  because $V, \varphi \in L^\infty(\T^d,\R)$. 

The multiplicative operator $A_0:=\varphi$ is self-adjoint on $L^2(\T^d,\C)$.
$H^2(\T^d,\C)$ is a common core of $A_{\tau}$ and $A_0$. For every $\psi \in H^2(\T^d,\C)$,
$\| (A_{\tau}-A_0)\psi \|_{L^2} =
\tau \|(\Delta_A-V)\psi\|_{L^2} \rightarrow 0$ as $\tau \to 0$.
Thus, for every $\psi \in L^2(\T^d,\C)$,
$\| (L_{\tau}-e^{i\varphi})\psi\|_{L^2}
= \| (e^{i A_{\tau}} - e^{i A_0})\psi \|_{L^2}
\to 0$ as $\tau \to 0$.
By Lemma \ref{lem:reachable-operators-bis}, this proves that 
$e^{i \varphi}$ is $L^2$-STAR.

\medskip

\noindent \emph{Step 2: We prove that, if $\varphi \in C^2(\T^d,\R)$ and $e^{i \lambda \varphi}$ is $L^2$-STAR for every $\lambda \in \R$ then $e^{-i|\nabla\varphi|^2}$ is $L^2$-STAR.}
Let $\tau>0$. By Lemma \ref{lem:reachable-operators-bis}, the operator
$$\widetilde{L}_{\tau}:=
e^{i \frac{\varphi}{\sqrt{\tau}}} 
e^{i\tau(\Delta_A-V)}
e^{-i \frac{\varphi}{\sqrt{\tau}}}$$
is $L^2$-approximately reachable in time $\tau^+$.

The operator $\tau(\Delta_A-V)$ is self-adjoint with domain $H^2(\T^d,\C)$.
The multiplicative operator $\varphi/\sqrt{\tau}$ is self-adjoint on $L^2(\T^d,\C)$. $e^{i\varphi/\sqrt{\tau}}$ is an isomorphism of $H^2(\T^d,\C)$ because $\varphi \in C^2(\T^d,\R)$.
Then, direct computations give
\begin{align*}
\widetilde{L}_{\tau} &= \exp \left( i\tau
e^{i \frac{\varphi}{\sqrt{\tau}}} 
(\Delta_A-V)
e^{-i \frac{\varphi}{\sqrt{\tau}}}  \right)\\
&= \exp \left( i \tau (\Delta_A-V)-\sqrt{\tau}(2\langle \nabla \varphi , \nabla \rangle + \Delta \varphi+2i\langle A,\nabla\varphi\rangle) - i |\nabla \varphi|^2 \right).
\end{align*}
The operator
$$\widetilde{A}_{\tau} := 
\tau
e^{i \frac{\varphi}{\sqrt{\tau}}} 
(\Delta_A-V)
e^{-i \frac{\varphi}{\sqrt{\tau}}}$$
is self-adjoint on the domain $H^2(\T^d,\C)$.
The multiplicative operator 
$\widetilde{A}_0:=- |\nabla \varphi|^2$ 
is self-adjoint on $L^2(\T^d,\C)$ because
$\nabla \varphi \in L^{\infty}(\T^d,\R)$.
$H^2(\T^d,\C)$ is a common core of $\widetilde{A}_{\tau}$ and $\widetilde{A}_0$.
For every $\psi \in H^2(\T^d,\C)$,
$$ \|(\widetilde{A}_{\tau}-\widetilde{A}_0) \psi \|_{L^2} = 
\|\tau (\Delta-V) \psi +i\sqrt{\tau}(2\langle \nabla \varphi , \nabla \rangle  + \Delta \varphi+2i\langle A,\nabla\varphi\rangle)) \psi \|_{L^2} 
\underset{\tau \to 0 }{\longrightarrow} 0.$$
Thus, for every $\psi \in L^2(\T^d,\C)$,
$$\| (\widetilde{L}_{\tau} - e^{-i|\nabla \varphi|^2})\psi \|_{L^2} = 
\| (e^{i \widetilde{A}_{\tau}} - e^{i A_0}) \psi \|_{L^2}
\underset{\tau \to 0 }{\longrightarrow} 0.$$
This proves that $e^{-i|\nabla \varphi|^2}$ is $L^2$-STAR.

\medskip

\noindent \emph{Step 3: Iteration.} We define by induction an increasing sequence of sets $(\mathcal{H}_{j})_{j\in\N}$: 
$\mathcal{H}_0$ is defined in Step 1 and, for every $j \in \N^*$,
$\mathcal{H}_j$ is the largest vector space whose elements can be written as 
$$  \varphi_0-\sum_{k=1}^N|\nabla\varphi_k|^2 ; N\in\N, \varphi_0,\dots,\varphi_N\in \mathcal{H}_{j-1}. $$
Let $\mathcal{H}_{\infty} := \cup_{j\in\N} \mathcal{H}_j$.
Thanks to Steps 1 and 2,
for every $\varphi \in \mathcal{H}_{\infty}$, the operator $e^{i \varphi}$ is $L^2$-STAR. Moreover, the proof of \cite[Proposition 2.6]{duca-nersesyan} shows that $\mathcal{H}_{\infty}$ contains any trigonometric polynomial. In particular, $\mathcal{H}_{\infty}$ is dense in $L^2(\T^d,\R)$.

%\noindent \emph{Step 4: Conclusion.} Let $\varphi \in L^2(\T^d,\R)$. There exists $(\varphi_n)_{n\in\N} \subset \mathcal{H}_{\infty}$ such that $\| \varphi_n -\varphi \|_{L^2} \rightarrow 0 $ as $n \to \infty$. Up to an extraction, one may assume that $\varphi_n \rightarrow  \varphi$ almost everywhere on $\T^d$, as $n \to \infty$. The dominated convergence theorem proves that, for every $\psi \in L^2(\T^d,\C)$,
%$\| (e^{i \varphi_n}-e^{i\varphi}) \psi \|_{L^2} \rightarrow 0$ as $n \to \infty$. Finally, Step 3 and Lemma \ref{lem:reachable-operators} prove that the operator $e^{i\varphi}$ is $L^2$-STAR.
\end{proof}

\subsection{On $\R^d$}

\begin{proposition} \label{Prop:DP}
Let $V$ satisfying (\ref{Hyp:V_transp}).
System \eqref{eq:line} satisfies the following property: 
\begin{itemize}
\item for every $\varphi \in L^\infty_{loc} (\R^d,\R)$ or $\varphi \in L^2 (\R^d,\R)$, the operator $e^{i \varphi}$ is $L^2$-STAR.
\item for every $u\in \R, j=1,\dots,d$, the operator $e^{u\partial_{x_j}}$ is $L^2$-STAR
\end{itemize}
\end{proposition}

\begin{proof}
\noindent \emph{Step 1: We prove that, for every $\varphi \in
\text{span}\{ x_1,\dots,x_d, e^{-|x|^2/2} \}$, the operator $e^{i \varphi}$ is $L^2$-STAR.} Let $\alpha \in\R^{d+1}$ and $\varphi: x \in \R^d \mapsto \alpha_1 x_1 + \dots + \alpha_d x_d + \alpha_{d+1} e^{-|x|^2/2}$. For any $\tau>0$, the operator 
$e^{i\tau(\Delta-V)+i\varphi}$ is $L^2$-exactly reachable in time $\tau$ because associated with the constant controls $u_j=-\alpha_j/\tau$ applied on the time interval $[0,\tau]$. The operator $\tau(\Delta_A-V)+\varphi$ is essentially self-adjoint on the domain $C^\infty_c(\R^d,\C)$, thus its closure $A_{\tau}$ is self-adjoint. The multiplicative operator $\varphi$ is self-adjoint on $\{ \psi \in L^2(\R^d,\C) ; \varphi \psi \in L^2(\R^d,\C) \}$. $C^\infty_c(\R^d,\C)$ is a common core of $A_{\tau}$ and $\varphi$. For every $\psi \in C^\infty_c(\R^d,\C)$, 
$\| (A_{\tau}-\varphi) \psi \|_{L^2} =
\tau \| (\Delta_A-V) \psi \|_{L^2} \to 0$ as $\tau \to 0$. Thus, for every $\psi \in L^2(\R^d,\C)$,
$\|(e^{i\tau(\Delta_A-V)+i\varphi} - e^{i\varphi})\psi  \|_{L^2} \to 0$ as $\tau \to 0$. Therefore, by Lemma \ref{lem:reachable-operators-bis}, the operator $e^{i \varphi}$ is $L^2$-STAR.

\medskip

\noindent \emph{Step 2: We prove that, for every $j \in \{1,\dots,d\}$ and $u \in \R$, the operator $e^{u \partial_{x_j}}$ is $L^2$-STAR.} Let 
$j \in \{1,\dots,d\}$, $u \in \R^*$. By Step 1 and Lemma \ref{lem:reachable-operators-bis}, for every $\tau>0$, the operator
$$L_{\tau}:= e^{\frac{iux_j}{2\tau}}e^{i\tau(\Delta_A-V)}e^{-\frac{iux_j}{2\tau}} $$
is $L^2$-approximately reachable in time $\tau^+$ .
The operator $\tau(\Delta_A-V)$ is essentially self-adjoint on $C^\infty_c(\R^d,\C)$ thus its closure $A$ is self-adjoint. The operator 
$B:=u x_j / (2\tau)$ is self-adjoint on $D(B):=\{ \psi \in L^2(\R^d,\C) ; x_j \psi \in L^2(\R^d,\C) \}$. 
$C^{\infty}_c(\R^d,\C)$ is a core of $A$.
The operator $e^{iB}$ is an isomorphism of $C^{\infty}_c(\R^d,\C)$. Then, direct computations show
\begin{equation}\label{eq:conjugate-dynamics}
L_{\tau} = \exp\left( i\tau
e^{\frac{iux_j}{2\tau}}
(\Delta_A-V)
e^{-\frac{iux_j}{2\tau}}
\right) = \exp \left(
i \tau (\Delta_A-V)+u (\partial_{x_j}+iA_j) - i \frac{u^2}{2\tau}
\right).
\end{equation}
Moreover, thanks to a global-phase change, we have that
$$L_{\tau}' := e^{i \frac{u^2}{2\tau}} L_{\tau}
=  \exp\left( 
i \tau (\Delta_A-V)+u( \partial_{x_j}+iA_j)
\right)$$
is also $L^2$-approximately reachable in time $\tau^+$.
The operator 
$$\tau
e^{\frac{iux_j}{2\tau}}
(\Delta_A-V)
e^{-\frac{iux_j}{2\tau}}
+ \frac{u^2}{2\tau}
=
\tau (\Delta_A-V) - i u \partial_{x_j} +uA_j
$$
is essentially self-adjoint on $C^\infty_c(\R^d,\C)$, thus its closure $A_{\tau}$ is self-adjoint. The operator
$A_0:= -i u \partial_{x_j}+uA_j$ is self-adjoint on
$D(A_0):=\{ \psi \in L^2(\R^d,\C) ; (-i\partial_{x_j}+A_j) \psi \in L^2(\R^d,\C) \}$. $C^\infty_c(\R^d;\C)$ is a common core of $A_{\tau}$ and $A_0$. For every $\psi \in C^{\infty}_c(\R^d,\C)$,
$\| (A_{\tau}-A_0)\psi \|_{L^2} = 
\tau \| (\Delta-V) \psi \|_{L^2} \to 0$ as $\tau \to 0$.
Thus, for every $\psi \in L^2(\R^d,\C)$, 
$\|(L_{\tau}'-e^{u (\partial_{x_j}+iA_j)})\psi \|_{L^2}
=
\| (e^{i A_{\tau}} - e^{i A_0}) \psi \|_{L^2}
\to 0$ as $\tau \to 0$. This proves that the operator $e^{u (\partial_{x_j}+iA_j)}$ is $L^2$-STAR.

\medskip

\noindent \emph{Step 3: We prove that, if $\varphi \in C^1(\R^d,\R)$ and $e^{i \lambda \varphi}$ is $L^2$-STAR for every $\lambda \in \R$ then the operator $e^{-i \partial_{x_j} \varphi}$ is $L^2$-STAR.} 
Let $\tau>0$. The assumption on $\varphi$ and Step 2 prove that the operator
$$\widetilde{L}_{\tau} := e^{i\frac{\varphi}{\tau}}e^{\tau(\partial_{x_j}+iA_j)}e^{-i\frac{\varphi}{\tau}}$$
is $L^2$-STAR. Direct computations show that
$$\widetilde{L}_{\tau}=\exp(\tau(\partial_{x_j}+iA_j)-i\partial_{x_j}\varphi), $$
hence, for every $\psi \in L^2(\R^d,\C)$,
$\| (\widetilde{L}_{\tau}-e^{-i\partial_{x_j}\varphi}) \psi \|_{L^2} \to 0$ as $\tau \to 0$. Finally,
the operator $e^{-i \partial_{x_j} \varphi}$ is $L^2$-STAR.

\medskip

\noindent \emph{Step 4: Iteration.} We define by induction an increasing sequence of sets $(\mathcal{H}_{j})_{j\in\N}$ by
$\mathcal{H}_0=\text{span}\{ e^{-|x|^2/2} \}$ and, for every $j \in \N^*$,
$$\mathcal{H}_j := \text{span}_{\R} \left\{  \varphi_0-\sum_{k=1}^d \partial_{x_k} \varphi_k ;  \varphi_0,\dots,\varphi_d \in \mathcal{H}_{j-1} \right\},$$
and $\mathcal{H}_{\infty} := \cup_{j\in\N} \mathcal{H}_j$.
Thanks to Steps 1 and 3,
for every $\varphi \in \mathcal{H}_{\infty}$, the operator $e^{i \varphi}$ is $L^2$-STAR. Moreover, by the proof of \cite[Lemma 5.2]{duca-pozzoli}, $\mathcal{H}_{\infty}$ is dense in $L^2(\R^d,\R)$ because it contains the linear combinations of Hermite functions.

\medskip

\noindent \emph{Step 5: Conclusion.} Let $\varphi \in L^2 (\R^d,\R)$. There exists $(\varphi_n)_{n\in\N} \subset \mathcal{H}_{\infty}$ such that $\| \varphi_n -\varphi \|_{L^2} \rightarrow 0 $ as $n \to \infty$. Up to an extraction, one may assume that $\varphi_n \rightarrow  \varphi$ almost everywhere on $\R^d$, as $n \to \infty$. The dominated convergence theorem proves that, for every $\psi \in L^2(\R^d,\C)$,
$\| (e^{i \varphi_n}-e^{i\varphi}) \psi \|_{L^2} \rightarrow 0$ as $n \to \infty$. Finally, Step 4 proves that the operator $e^{i\varphi}$ is $L^2$-STAR.

To conclude, notice that if $\varphi\in L^\infty_{loc}$, then $e^{i\varphi}$ is STAR. Indeed, given $\psi_0\in L^2(\R^d,\C)$,\\$\varepsilon>0$, there exists $R>0$ such that $\|e^{i\varphi}\psi_0-e^{i\varphi \chi_{[-R,R]^d}}\psi_0\|_{L^2}\leq \varepsilon/2. $ Now, $\varphi \chi_{[-R,R]^d}\in L^2(\R^d,\R)$, hence $e^{i\varphi \chi_{[-R,R]^d}}$ is STAR: so, there exists a control $u\in PWC[0,\tau]$ with $\tau<\epsilon$ such that $\|\psi(\tau,u,\psi_0)-e^{i\varphi\chi_{[-R,R]^d}}\psi_0\|_{L^2}<\varepsilon/2$. By triangular inequality, we see that $\|\psi(\tau,u,\psi_0)-e^{i\varphi}\psi_0\|_{L^2}<\varepsilon$, hence $e^{i\varphi}$ is STAR.

In particular, since $A_j\in C^1\subset L^\infty_{loc}$, the operator $(e^{-iuA_j/n}e^{(u\partial_{x_j}+iuA_j)/n})^n$ is STAR, so is its strong limit as $n\to \infty$, $e^{u\partial_{x_j}}$.
\end{proof}

The proof of Theorem \ref{thm:examples} then follows exactly as in \cite{beauchard-Pozzoli2}.
%%%%%%%%%%%%%%%%%%%%%%%%%%%%%%%%%%%%%%%%%%%%%%%%%%%%%%%%%%%%%%%

\medskip

\textbf{Acknowledgments.}
The author acknowledges financial support from grants ANR-25-CE40-4062 (project QUEST), ANR-24-CE40-3008-01 (Project QuBiCCS) and ANR-11-LABX-0020 (Labex Lebesgue), as well as from the CNRS through the MITI interdisciplinary programs.

\bibliographystyle{siamplain}
\bibliography{references}

\end{document}